\documentclass[12pt]{amsart}

\usepackage{amssymb}
\usepackage[margin=1.22in]{geometry}
\usepackage{amsmath,amsthm,url}
\usepackage{microtype}
\usepackage{enumitem}
\makeatletter
\renewcommand{\pod}[1]{\mathchoice
  {\allowbreak \if@display \mkern 18mu\else \mkern 8mu\fi (#1)}
  {\allowbreak \if@display \mkern 18mu\else \mkern 8mu\fi (#1)}
  {\mkern4mu(#1)}
  {\mkern4mu(#1)}
}
\DeclareMathAlphabet{\curly}{U}{rsfs}{m}{n}

\newtheorem{thm}{Theorem}[section]

\newtheorem{lem}[thm]{Lemma}

\newtheorem{prop}[thm]{Proposition}

\newtheorem*{theorem*}{Theorem}
\theoremstyle{remark}\newtheorem*{remark}{Remark}

\newcommand{\justif}[2]{&{#1}&\text{#2}}
\newcommand{\Nm}{\mathrm{Nm}}
\newcommand{\Tr}{\mathrm{Tr}}
\newcommand{\Oo}{\mathcal{O}}
\newcommand{\pp}{\mathfrak{p}}

\newcommand\sumprime{\sideset{}{^{'}}{\sum}}

\newlist{owndesc}{description}{1}  
\setlist[owndesc]{leftmargin=0.5cm,labelsep=2cm} 

\newcommand{\leg}[2]{\genfrac{(}{)}{}{}{#1}{#2}}
\newcommand\Li{\mathrm{Li}}

\newcommand\F{\mathbb{F}}
\newcommand\Z{\mathbb{Z}}

\newcommand\Q{\mathbb{Q}}

\newcommand\lcm{\mathrm{lcm}}

\newcommand\Ll{\curly{L}}

\renewcommand{\phi}{\varphi}

\newcommand{\mm}{\mathfrak{m}}
\newcommand{\nn}{\mathfrak{n}}
\newcommand{\ff}{\mathfrak{f}}
\renewcommand{\aa}{\mathfrak{a}}
\newcommand{\bb}{\mathfrak{b}}
\newcommand{\Lcal}{\curly{L}}
\newcommand{\Cc}{\curly{C}}
\newcommand{\Ss}{\curly{S}}
\newcommand{\qq}{\mathfrak{q}}

\begin{document}
\author{Tristan Freiberg}
\address{Department of Mathematics\\ University of Missouri\\Columbia, MO 65211, USA}
\email{freibergt@missouri.edu}

\author{Paul Pollack}
\address{Department of Mathematics\\ University of Georgia\\Athens, GA 30602, USA}
\email{pollack@uga.edu}

\title[The first invariant factor for reductions of CM elliptic curves]{The average of the first invariant factor for reductions of CM elliptic curves mod $p$}
\begin{abstract}\noindent Let $E/\Q$ be a fixed elliptic curve. For each prime $p$ of good reduction, write $E(\F_p) \cong \Z/d_p \Z \oplus \Z/e_p \Z$, where $d_p \mid e_p$. Kowalski proposed investigating the average value of $d_p$ as $p$ runs over the rational primes. For CM curves, he showed that $x\log\log{x}/\log{x} \ll \sum_{p \le x} d_p \ll x\sqrt{\log{x}}$. It was shown recently by Felix and Murty that in fact $\sum_{p \le x} d_p$ exceeds any constant multiple of $x\log\log{x}/\log{x}$, once $x$ is sufficiently large. In the opposite direction, Kim has shown that the expression $x\sqrt{\log{x}}$ in the upper bound can be replaced by $x\log\log{x}$. In this paper, we obtain the correct order of magnitude for the sum: $\sum_{p \le x} d_p \asymp x$ for all large $x$.
\end{abstract}

\maketitle

\section{Introduction} Let $E/\Q$ be a fixed elliptic curve. For each rational prime $p$ of good reduction, there are uniquely defined natural numbers $d_p$ and $e_p$ with $E(\F_p) \cong \Z/d_p \Z \oplus \Z/e_p\Z$. From a statistical point of view, it is natural to inquire about the behavior of $d_p$ and $e_p$ as $p$ varies. This is all the more true given that $d_p$ and $e_p$ have arithmetic significance: $d_p$ is the largest integer prime to $p$ for which all of the $d$-torsion is rational over $\F_p$, and $e_p$ is the largest order of any element of $E(\F_p)$. Note that the sizes of $d_p$ and $e_p$ are closely intertwined, since $d_p e_p = \#E(\F_p) \in [(\sqrt{p}-1)^2, (\sqrt{p}+1)^2]$ by a celebrated theorem of Hasse.

For notational convenience, set $d_p=e_p=0$ when $E$ has bad reduction at $p$.

Responding to a suggestion of Silverman, Freiberg and Kurlberg \cite{FK14} investigated the average size of $e_p$. They showed that as $x\to\infty$, one has $\sum_{p\le x} e_p \sim c_E \Li(x^2)$ for a certain constant $c_E \in (0,1)$. Their result is unconditional if $E$ has CM and conditional on the Generalized Riemann Hypothesis otherwise.

It is a simple consequence of the prime number theorem that $\sum_{p\le x} p \sim \Li(x^2)$. Keeping in mind that $d_p e_p \sim p$, the result of Freiberg and Kurlberg suggests that $d_p$ is usually quite small. In fact, Duke \cite{duke03} has shown that for any function $\xi(p)\to\infty$, one has $d_p < \xi(p)$ for asymptotically 100\% of primes $p$. (Again, GRH is assumed here unless $E$ has CM.) Duke's result tells us about the normal size of $d_p$. What about the average size?

In fact, the problem of determining the average order of $d_p$ was proposed by Kowalski already in 2000. For reasons explained in \cite[\S3.2]{kowalski06}, it is natural to conjecture that as $x\to\infty$, $\sum_{p \le x} d_p$ is $\sim c_E' X$ when $E$ has CM and $\sim c_E' \Li(x)$ otherwise, where $c_E' > 0$. These conjectures remain open, even under GRH.

There has been only meager progress towards Kowalski's conjectures in the case when $E$ does not have complex multiplication. In what follows, we restrict our discussion to the CM case. There Kowalski showed that for large $x$,
\[ \frac{x \log\log{x}}{\log{x}} \ll \sum_{p \le x} d_p \ll x\sqrt{\log{x}}; \]
moreover, under GRH, the sum is $\gg x$. Unconditionally, Felix and Murty \cite{FM13} showed that for any $A$ and all sufficiently large $x$, we have $\sum_{p \le x} d_p > A x\log\log{x}/\log{x}$. In the opposite direction, Kim showed (among other things) that in the upper bound, the expression $x\sqrt{\log{x}}$ can be replaced by $x\log\log{x}$ \cite{kim14}.

In this paper, we establish the correct order of magnitude for the partial sums of $d_p$.
\begin{thm}\label{thm:main} Let $E/\Q$ be an elliptic curve with complex multiplication. Then
\[ \sum_{p \le x} d_p \ll x; \]
here the implied constant is absolute. Moreover, for $x > x_0(E)$,
\[ \sum_{p \le x} d_p \gg_{E} x. \]
\end{thm}

If one replaces $d_p$ with $d_p^{\alpha}$, for a fixed $\alpha \in (0,1)$, then Felix and Murty (op.\ cit.) have exhibited an asymptotic formula for the partial sums, conditional on GRH. The main term in their formula has the shape $c_{\alpha, E} \Li(x)$, so that $d_p^{\alpha}$ is bounded on average. Hence $\alpha=1$ is a transition point, since, by our main theorem, $d_p$ itself is $\asymp \log{x}$ on average over $p \le x$.
%

The proofs of the upper and lower bounds are based on distinct principles. Hence, the upper bound is treated in \S\ref{sec:upper} while the lower bound is treated separately in \S\S \ref{sec:lower1} and \ref{sec:lower2}.

\subsection*{Notation} We use the letter $K$ for an algebraic number field. We let $d_K$ denote the absolute value of the (absolute) discriminant of $K$. $\Oo_K$ denotes the the ring of integers of $K$. For $\alpha \in K$, we write $\Nm(\alpha)$ for the norm of $\alpha$ and $\Tr(\alpha)$ for its trace. For an ideal $\aa$ of $\Oo_K$, we let $\Nm(\aa):= \#\Oo_K/\aa$ and we let $\Phi(\aa):= \#(\Oo_K/\aa)^{\times}$. For $\alpha \in \Oo_K$, we let $\Phi(\alpha)$ be the $\Phi$ function applied to the principal ideal $(\alpha)$.

The case when $K$ is an imaginary quadratic extension of $\Q$ plays a special role for us. If $K= \Q(\sqrt{g})$ where $g < 0$ is squarefree, then we set $\omega = \sqrt{g}$ if $g \equiv 2\text{ or }3\pmod{4}$ and $\omega = \frac{1+\sqrt{g}}{2}$ otherwise. Thus, $1, \omega$ form an integral basis of $K$. We use the symbol $\Oo$ to denote a possibly nonmaximal order of $K$.

The letters $\ell$ and $p$ are reserved for rational primes. We use $\mathfrak{p}$ for a maximal ideal of $\Oo_K$. If $\mathfrak{p}$ lies over the rational prime $p$, then $\deg(\pp)$ denotes the degree of $\Oo_K/\pp$ over $\Z/p\Z$.

\section{The upper bound in Theorem \ref{thm:main}}\label{sec:upper}
\subsection{Preliminaries}

We begin by recording an important alternative description of $d_p$ in the case when $p$ is of good ordinary reduction. Let us fix notation. Suppose that $E/\Q$ is an elliptic curve with complex multiplication by an order $\Oo$ in the imaginary quadratic field $K$. Since $E$ is defined over $\Q$, the field $K$ is one of the nine imaginary quadratic fields of class number $1$, and $\Oo$ is one of the thirteen imaginary quadratic orders of class number $1$. (See \cite[p.\ 483]{silverman94} for a list of these orders along with the corresponding curves.) A rational prime $p$ of good reduction is an ordinary prime if and only if $p$ splits completely in $K$; in that case, as long as $p$ does not divide the conductor of $\Oo$, we can identify $\Oo$ with the ring of endormorphisms of the reduced curve $E$ mod $p$. (For these last two statements, see \cite[Theorem 12, p.\ 182]{lang87}.) Since our orders $\Oo$ all have conductor at most $3$, we can make this identification whenever $p > 3$.

\begin{lem}\label{lem:dpformula} Let $p > 3$ be a prime at which $E$ has good ordinary reduction. Let $\pi_p \in \Oo$ be the Frobenius endormorphism of the reduced curve. Then $d$ divides $d_p$ if and only if $\pi_p \equiv 1\pmod{d}$ in $\Oo$.
\end{lem}

\begin{proof} For integers $d$ coprime to $p$, we have
\begin{align*} d \mid d_p &\Longleftrightarrow E[d](\overline{\F_p}) \subset E(\F_p)  \justif{\quad}{(see \cite[Lemma 2.3(i)]{kowalski06})}\\ &\Longleftrightarrow \pi_p \equiv 1\pmod{d} \text{ in $\Oo$}  \justif{\quad}{(see \cite[Lemma 2.6]{kowalski06})}.\end{align*}
Now suppose that $p \mid d$. We will show that we have neither $d \mid d_p$ nor $\pi_p\equiv 1\pmod{d}$. Since $d_p^2 \mid \#E(\F_p)$ and $\#E(\F_p) \le (\sqrt{p}+1)^2$, we have $d_p \leq \sqrt{p}+1 < p$. Hence, $d_p$ is not a multiple of $p$ and so not a multiple of $d$. Since $\#E(\F_p) = \Nm(\pi_p-1)$, if $\pi_p \equiv 1\pmod{d}$, then $p^2 \mid d^2 \mid \#E(\F_p)$. This leads to the absurd inequality $p^2 \le \#E(\F_p) \le (\sqrt{p}+1)^2$.
\end{proof}

We also require two items from the analytic toolchest. The first is a Brun--Titchmarsh inequality for imaginary quadratic fields. This appears as \cite[Lemma 2.5]{pollack14}, where it is deduced from a Brun--Titchmarsh theorem for prime ideals established by Hinz and Lodemann \cite[Theorem 4]{HL94}. Let
\[ \pi(x;\mu,\alpha) := \#\{\text{prime elements }\pi: \Nm(\pi) \le x, \pi \equiv \alpha\pmod{\mu}\}. \]

\begin{lem}\label{lem:BT} Let $x\ge 3$. Suppose that $\mu, \alpha \in \Oo_K$ generate comaximal ideals. If $\Nm(\mu) < x$, then
\[ \pi(x;\mu,\alpha)\ll \frac{x}{\Phi(\mu) \log\frac{x}{\Nm(\mu)}}. \]
The implied constant may depend on $K$.
\end{lem}

\begin{remark} In the statement of \cite{pollack14}, $K$ is assumed to be of class number $1$. In fact, the proof indicated in \cite{pollack14} goes through without any restriction on the class number of $K$. Note that if we assume $K$ has bounded class number, then there are only finitely many possibilities for $K$, and so the implied constant of the lemma can be chosen uniformly.
\end{remark}

The following lemma, which is a weakened form of a theorem of Halberstam and Richert \cite{HR79} (compare with \cite[Theorem 3.2, p.\ 58]{SS94}), is a versatile upper bound result for mean values of multiplicative functions.

\begin{lem}\label{lem:HR} Let $\lambda_1, \lambda_2$ be positive constants with $\lambda_2 < 2$. Suppose that $g$ is a nonnegative-valued multiplicative function with $g(p^k) \le \lambda_1 \lambda_2^k$ for all primes $p$ and all positive integers $k$. Then
\[ \sum_{n \le x} g(n) \ll_{\lambda_1, \lambda_2} x \prod_{p \le x}\left(1-\frac{1}{p}\right)\left(1+\frac{g(p)}{p} + \frac{g(p^2)}{p^2} + \dots\right). \]
\end{lem}

\subsection{The proof proper}
We begin by discarding from $\sum_{p \le x} d_p$ all supersingular primes $p$. It is simple to show that for each supersingular prime, one has $d_p \le 2$, and so these terms contribute only $O(x/\log{x})$. (For details, see the proof of \cite[Corollary 6.2]{kowalski06}.)

Let $\sum_{p}^{'}$ denote a sum restricted to primes $p$ of good ordinary reduction. To prove the upper bound in Theorem \ref{thm:main}, it suffices to show that $\sum_{3 < p \le x}^{'} d_p \ll x$. Recall that $\phi(m) = \sum_{d \mid m} \phi(d)$ for every positive integer $m$. Since $d_p \le \sqrt{p}+1 \le 2\sqrt{x}$ for all $p \le x$,
\begin{align}\notag\sumprime_{3 < p \le x} d_p &= \sumprime_{3 < p \le x} \sum_{d \mid d_p} \phi(d)\\
 &=\sum_{d \le 2\sqrt{x}} \phi(d) \sumprime_{\substack{3< p \le x \\ d \mid d_p}} 1.  \label{eq:analogue} \end{align}
We first show that those $d \le x^{1/3}$ make a contribution to \eqref{eq:analogue} of size $O(x)$. This estimate is already implicit in the works of both Kowalski and Kim, but we include the argument for completeness.

For each prime $p$ counted in the inner sum of \eqref{eq:analogue}, the Frobenius element $\pi_p\in \Oo$ is a prime of $\Oo_K$ with $\pi_p \equiv 1\pmod{d}$ and $\Nm(\pi_p)=p$. So by Lemma \ref{lem:BT}, that sum is $\ll \frac{1}{\Phi(d)} \frac{x}{\log{x}}$ uniformly for $d < x^{1/3}$, and thus the right-hand side of \eqref{eq:analogue} is
\[ \ll \frac{x}{\log{x}} \sum_{d \le x^{1/3}} \frac{\phi(d)}{\Phi(d)}.\]
Writing $\Delta$ for the discriminant of $K$, we have
\[ \Phi(d) = d^{2} \prod_{\ell \mid d}\bigg(1-\frac{1}{\ell}\bigg)\bigg(1-\frac{\leg{\Delta}{\ell}}{\ell}\bigg) \ge \phi(d)^2 \]
for all $d$. Thus, \eqref{eq:analogue} is $\ll \frac{x}{\log{x}}\sum_{d \le x^{1/3}} \frac{1}{\phi(d)} \ll \frac{x}{\log{x}} \cdot \log{x} = x$, as claimed.

Handling those values of $d$ with $x^{1/3} < d \le 2 \sqrt{x}$ requires a different strategy. Let $I_j= (2^j x^{1/3}, 2^{j+1} x^{1/3}]$, where $j$ runs over all nonnegative integers with $2^j x^{1/3} < 2\sqrt{x}$. We consider the contribution to the right-hand side of \eqref{eq:analogue} from $d$ in each $I_j$.

Using Lemma \ref{lem:dpformula}, we see that
\begin{equation}\label{eq:toinvert} \sumprime_{\substack{3< p \le x \\ d \mid d_p}} 1 \le \sum_{\substack{X, Y \in \Z \\ \Nm((X+Y\omega)d+1) \le x\\
\text{ and prime}}} 1. \end{equation}
If $\Nm((X+Y\omega)d+1) \le x$, then $|(X+Y\omega)d| \le 1+\sqrt{x} \le 2\sqrt{x}$. Hence, assuming $d \in I_j$, we must have
\[ \Nm(X+Y\omega) \le 4x/d^2 \le 2^{2-2j} x^{1/3}. \] Moreover, if $\Nm((X+Y\omega)d+1)$ is prime, then $Y \ne 0$. Inserting \eqref{eq:toinvert} back into \eqref{eq:analogue} and reversing the order of summation reveals that the $d \in I_j$ contribute at most
\begin{multline} \sum_{\substack{X, Y \in \Z,~Y\ne 0 \\ \Nm(X+Y\omega) \le 2^{2-2j} x^{1/3}}} \sum_{\substack{d \in I_j \\ \Nm((X+Y\omega)d+1)\text{ prime}}} \phi(d) \\ \ll 2^j x^{1/3} \sum_{\substack{X, Y \in \Z,~ Y\ne 0 \\ \Nm((X+Y\omega)) \le 2^{2-2j} x^{1/3}}} \sum_{\substack{d \in I_j \\ \Nm((X+Y\omega)d+1)\text{ prime}}} 1 \label{eq:almostsieving}.\end{multline}

The remaining sum on $d$ can be estimated by Brun's sieve. For each $X, Y \in \Z$ with $\Nm(X+Y\omega) \le 2^{2-2j} x^{1/3}$ and $Y\ne 0$, put
\begin{align*} F(T) = \Nm(X+Y\omega) \cdot T^2 + \Tr(X+Y\omega)\cdot T + 1 \in \Z[T].
\end{align*}
(Of course, $F$ depends on $X$ and $Y$ but we suppress this.) Then $F$ is a quadratic polynomial with discriminant $Y^2 \Delta$, where as above $\Delta$ denotes the discriminant of $K$. The final sum on $d$ in \eqref{eq:almostsieving} counts the number of $d \in I_j$ for which $F(d)$ is prime. By the fundamental lemma of the sieve (see \cite[Theorem 2.2, p.\ 68]{HR74}), the number of these $d$ is
\[ \ll 2^j x^{1/3} \prod_{\ell \le x}\left(1-\frac{\rho(\ell)}{\ell}\right), \]
where $\rho(\ell)$ counts the number of roots of $F$ modulo $\ell$. Put $D = 2\cdot \Nm(X+Y\omega) \cdot |Y|$. For $\ell$ not dividing $D$, we have $\rho(\ell) = 1 + \leg{\Delta}{\ell}$. Consequently,
\[ \prod_{\ell \le x}\bigg(1-\frac{\rho(\ell)}{\ell}\bigg) \ll \bigg(\frac{D}{\phi(D)}\bigg)^2 \cdot \prod_{\ell \le x}\bigg(1-\frac{\leg{\Delta}{\ell}}{\ell}\bigg) \prod_{\ell \le x}\bigg(1-\frac{1}{\ell}\bigg). \]
The first right-hand product over $\ell$ is $O(1)$, since the product extended to infinity converges to $L(1,\leg{\Delta}{\cdot})^{-1}$. (Note that only finitely many values of $\Delta$ are possible, and so the $O$-constant is absolute.) The second product on $\ell$ is $\ll (\log{x})^{-1}$. Thus,
\[ \sum_{\substack{d \in I_j \\ \Nm((X+Y\omega)d+1)\text{ prime}}} 1 \ll \frac{2^j x^{1/3}}{\log{x}} \frac{D^2}{\phi(D)^2},  \]
and so the right-hand side of \eqref{eq:almostsieving} is
\begin{equation}\label{eq:Dineq} \ll \frac{2^{2j} x^{2/3}}{\log{x}} \sum_{\substack{X, Y \in \Z,~Y\ne 0 \\ \Nm(X+Y\omega) \le 2^{2-2j} X^{1/3}}} \frac{D^2}{\phi(D)^2}. \end{equation}
We now show that $\frac{D^2}{\phi(D)^2}$ is bounded on average over $X$ and $Y$. Notice that
\[\frac{D^2}{\phi(D)^2} \ll \frac{\Nm(X+Y\omega)^2}{\phi(\Nm(X+Y\omega))^2} \frac{Y^2}{\phi(|Y|)^2}.\]
Applying the Cauchy--Schwarz inequality, we deduce that the sum on $D$ in \eqref{eq:Dineq} is
\[ \ll\Bigg(\sum_{\substack{X, Y \in \Z,~Y\ne 0 \\ \Nm(X+Y\omega) \le 2^{2-2j} x^{1/3}}} \frac{\Nm(X+Y\omega)^4}{\phi(\Nm(X+Y\omega))^4}\Bigg)^{1/2} \Bigg(\sum_{\substack{X, Y \in \Z,~Y\ne 0 \\ \Nm(X+Y\omega) \le 2^{2-2j} x^{1/3}}} \frac{Y^4}{\phi(|Y|)^4}\Bigg)^{1/2}.\]

The second sum on $X$ and $Y$ is the easier of the two to handle. The conditions on $X$ and $Y$ imply that $|X|$ and $|Y|$ are both $O(2^{-j} x^{1/6})$. We now use the known estimate
\begin{equation}\label{eq:schur} \sum_{m \le t} \frac{m^4}{\phi(m)^4} \ll t \qquad\text{(for all $t\ge 0$)}\end{equation}
to deduce --- summing first on $Y$ and then on $X$ --- that this second  sum is $O(2^{-2j} x^{1/3})$. The estimate \eqref{eq:schur} could be proved by applying Lemma \ref{lem:HR}; we omit this, as we shall see a similar but slightly more intricate calculation momentarily. In fact, \eqref{eq:schur} is classical and a more general result was known already to Schur (see \cite[p.\ 214]{elliott79} for a discussion).

Turning to the first sum, we let $m = \Nm(X+Y\omega)$. Since $K$ has class number $1$, the number of $X, Y\in \Z$ with $\Nm(X+Y\omega)=m$ is given by
\[ r(m):= w \sum_{e \mid m} \leg{\Delta}{e}, \]
where $w$ is the number of roots of unity in $K$. (Cf. \cite[Theorem 148, p.\ 179]{hecke81}. Without using that $K$ has class number $1$, we could still conclude that $r(m)$ is an upper bound on the number of pairs $X,Y$, which would suffice below.) Put $r^{\ast}(m) = r(m)/w$ and note that $r^{\ast}$ is a multiplicative function taking only nonnegative values. We can bound the first sum on $X, Y$ by
\[ w\sum_{m \le 2^{2-2j} x^{1/3}} r^{\ast}(m) \frac{m^4}{\phi(m)^4}. \]
Since $r^{\ast}(m) \le \tau(m)$, it is easy to see that the hypotheses of Lemma \ref{lem:HR} are satisfied for $g(n):= r^{\ast}(n) \frac{n^4}{\phi(n)^4}$. Applying that lemma shows that the last displayed quantity is
\begin{equation}\label{eq:almostdone} \ll 2^{-2j} x^{1/3} \prod_{p \le 2^{2-2j} x^{1/3}} \bigg(1-\frac{1}{p}\bigg)\bigg(1 + \sum_{k=1}^{\infty}\frac{r^\ast(p^k) (p/\phi(p))^4}{p^k}\bigg). \end{equation}
Now $1 + \sum_{k=1}^{\infty}\frac{r^\ast(p^k) (p/\phi(p))^4}{p^k} = 1 + \frac{1}{p} + \frac{\leg{\Delta}{p}}{p} + O(1/p^2)$, so that
\[ \bigg(1-\frac{1}{p}\bigg)\bigg(1 + \sum_{k=1}^{\infty}\frac{r^\ast(p^k) (p/\phi(p))^4}{p^k}\bigg) = 1 + \frac{\leg{\Delta}{p}}{p} + O(1/p^2). \]
Since $\sum_{p} \leg{\Delta}{p}/p$ converges, we see now that the product in \eqref{eq:almostdone} is $\ll 1$, and so \eqref{eq:almostdone} itself is $O(2^{-2j} x^{1/3})$. Assembling the estimates of this paragraph and the last yields
\[ \sum_{\substack{X, Y \in \Z,~Y\ne 0 \\ \Nm(X+Y\omega) \le 2^{2-2j} X^{1/3}}} \frac{D^2}{\phi(D)^2} \ll 2^{-2j} x^{1/3}. \]
Now from \eqref{eq:Dineq}, we see that the right-hand side of \eqref{eq:almostsieving} is $O(x/\log{x})$.

It remains to sum this upper bound over the possible values of $j$. There are only $O(\log{x})$ of these, leading to a final upper bound of $O(x)$, as desired.

\section{Technical preliminaries for the proof of the lower bound}\label{sec:lower1}

The proof of the lower bound half of Theorem \ref{thm:main} requires us to recall certain results from the literature on the equidistribution of primes in ray class groups. Our main reference for this material is the paper of Weiss \cite{weiss83}, where Linnik's fundamental result on the least prime in a progression is generalized to arbitrary algebraic number fields.

\subsection{Background and notation} Let $K$ be an algebraic number field. (We do not assume to begin with that $K$ is imaginary quadratic, though in our application this will be the case.) Suppose $[K:\Q]=n = r_1 + 2r_2$, where $r_1$ is the number of real embeddings of $K$ and $r_2$ the number of pairs of complex conjugate embeddings.  If $\mm$ is a (nonzero) ideal of $\Oo_K$, let $I(\mm)$ denote the group of fractional ideals relatively prime to $\mm$, and let $P_{\mm}$ be the subgroup defined by
\[ P_{\mm} := \{\alpha \Oo_K: \alpha \in K^{\times}, \alpha \text{ totally positive}, \alpha\equiv 1~\mathrm{mod}^{*}~m\}. \]
The \emph{narrow class group} mod $\mm$ is the quotient $I(\mm)/P_{\mm}$. We say $\aa, \bb \in I(\mm)$ are \emph{strictly equivalent} modulo $\mm$, and write $\aa \sim \bb\pmod{\mm}$, if $\aa$ and $\bb$ represent the same coset modulo $P_{\mm}$. A \emph{(Dirichlet) character modulo $\mm$} is a character of the finite abelian group $I(\mm)/P_{\mm}$. By a \emph{congruence class group} mod $\mm$, we mean a subgroup $H$ of $I(\mm)$ containing $P_{\mm}$.

Whenever $\mm \mid \nn$, there is a canonical surjection $I(\nn)/P_\nn \twoheadrightarrow I(\mm)/P_{\mm}$. Composing with a character $\chi$ mod $\mm$ yields a character $\chi'$ mod $\nn$. We say $\chi$ \emph{induces} $\chi'$. Similarly, if $H$ is a congruence class group mod $\mm$, taking the preimage of $H$ under the specified surjection yields an \emph{induced subgroup} $H'$ mod $\nn$. The \emph{conductor of $\chi$}, denoted $\ff_{\chi}$, is the smallest modulus (with respect to the partial order by divisibility) from which $\chi$ can be induced. We similarly define the conductor $\ff_{H}$ of a congruence class group $H$. One can show that
\[ \ff_{H} = \lcm\{\ff_{\chi}: \chi(H)=1\}. \]
For each character $\chi$, we set \[ d_{\chi} := d_K \cdot \Nm(\ff_{\chi}). \] For each congruence class group $H$, we write
\[ h_{H} := \#I(\mm)/H \quad\text{and}\quad d(H) := \max\{d_{\chi}: \chi(H)=1\}. \]

Let $\chi$ be a character modulo $\mm$. For $\sigma:=\Re(s) > 1$, we define the $L$-series $L(s,\chi) = \sum_{\aa} \chi(\aa) \cdot \Nm(\aa)^{-s}$.  Then $L(s,\chi)$ has an analytic continuation to the entire complex plane, except for a simple pole at $s=1$ when $\chi$ is principal. The \emph{nontrivial zeros} of $L(s,\chi)$ are those zeros belonging to the strip $0 < \sigma < 1$.

\subsection{A theorem of Weiss} The goal of this section is to describe a variant of  Weiss's theorem. In the following results, $c_1, c_2, \dots$ denote absolute positive constants. For the reader's convenience, we have used the same numbering as in Weiss's paper.

\begin{prop}\label{prop:except0} For $Q\ge 1$ and $T\ge 1$, put $\Lcal = \log(QT^n)$. Suppose that $\Lcal$ exceeds a certain absolute constant. There is at most one primitive character $\chi$ with $d_{\chi} \le Q$ for which $L(s,\chi)$ has a zero $\sigma+it$ with
\[ \sigma \ge 1-c_1 \Lcal^{-1} \quad\text{and}\quad |t| \le T. \]
\end{prop}
\noindent For the proof, see \cite[Theorem 1.9]{weiss83}. If the character $\chi$ of the last proposition exists, it is called the \emph{exceptional character} with respect to $Q$ and $T$. Similarly, $\ff_{\chi}$ is called the \emph{exceptional modulus} and $\sigma+it$ is called the \emph{exceptional zero}.

The following result is a short interval variant of Linnik's theorem, for prime ideals.

\begin{thm}[Weiss]\label{thm:weiss} Let $H\bmod{\mm}$ be a congruence subgroup and let $\Cc$ be a coset of $I(\mm)/H$. Define
\[ \pi_\Cc(x,\delta) := \sum_{\substack{\pp \in \Cc \\ x (1-\delta) < \Nm(\pp) < x \\ \deg(\pp)=1}} 1. \] Suppose that $Q \ge 1$ and that
\begin{equation}\label{eq:deltaineq} 0 < \delta \le c_{10} h_H^{-\frac{1}{2n}} Q^{-\frac{1}{2n}}. \end{equation}
Let $\nn$ be the product of the primes dividing $\mm$ but not $\ff_H$, and suppose that
\begin{equation}\label{eq:xineq} x \ge \max\{(\log \Nm(\nn))^2, (30 n Q^{\frac{1}{2n}} \delta^{-1})^{c_{11} n}\}. \end{equation}
With $T = (4(2n+3) \delta^{-1})^{2}$, assume that the exceptional character corresponding to $Q$ and $T$ --- if it exists --- does not induce a character $\chi$ {\rm{mod}} $\mm$ having $\chi(H)=1$. Then
\[ \pi_\Cc(x,\delta) \gg n^{-1} \cdot\frac{\delta x}{h_H \log{x}}. \]
Here the implied constant is absolute.
\end{thm}

\begin{proof} This follows from making small modifications in Weiss's proof of his Theorem 5.2 \cite{weiss83}. We now describe the necessary changes. We assume the reader has Weiss's paper in front of them for comparison.

\begin{owndesc}
\item[Variation in hypotheses] \hfill \\
In Weiss's version, $Q$ is immediately set equal to $d(H)$. Correspondingly, when Weiss states his assumptions on $x$ and $\delta$, he has $d(H)$ where we have $Q$. However, his arguments all go through under the hypothesis that $Q \ge d(H)$, if the conditions on $x$ and $\delta$ are stated as above.

\item[Variation in the definition of $\pi_{\Cc}(x,\delta)$]\hfill \\
In Weiss's statement, $\pi_\Cc(x,\delta)$ counts primes $\pp$ with $x < \Nm(\pp) < x(1+\delta)$, rather than $\pp$ satisfying $x (1-\delta) < \Nm(\pp) < x$. This appears to be a minor oversight, stemming from the incorrect claim at the bottom of p.\ 89 that $ye^{kA^{-1}} = xe^{\delta/2} \le x(1+\delta)$. In fact, the weights $H_k(y/\Nm(\pp))$ are only nonzero when $e^{-kA^{-1}} < \Nm(\pp)/y < e^{kA^{-1}}$ (by \cite[Lemma 3.2(a)]{weiss83}). Since
\[ x = y e^{kA^{-1}} \quad\text{and}\quad ye^{-kA^{-1}} = xe^{-\delta/2} > x(1-\delta), \] the counting function $\pi_{\Cc}(x,\delta)$ ought instead to be defined as above.

\item[Variation in the final lower bound on $\pi_{\Cc}(x,\delta)$]\hfill \\ Most significantly, the claimed lower bound on $\pi_\Cc(x,\delta)$ in \cite[Theorem 5.2]{weiss83} is quite a bit weaker than what we have asserted. This is because Weiss does not make any assumption on the (non)existence of exceptional characters.

To obtain the lower bound claimed in our Theorem \ref{thm:weiss}, we proceed as follows. From the first and last displayed equations on Weiss's p.\ 89,
\begin{multline*} 10n \frac{h_H \log{x}}{\delta x} \cdot \pi_{\Cc}(x,\delta) \ge 1 - \sum_{\chi(H)=1} \sum_{\rho_{\chi}} |h_k(\rho_{\chi}-1) y^{\rho_{\chi}-1}| \\ + O(h_H y^{c_6-1} \cdot T \log(d(H) T^n)) + O(h_H A^k T^{1-k} \cdot \log(d(H) T^n)). \end{multline*}
As argued at the top of p.\ 90, the second error term dominates if $c_{11}$ is chosen sufficiently large (as we may assume).

If the exceptional zero $\rho_{\ast}$ exists, then the argument at the top of p.\ 90 shows that the second error term is $O(\delta \Delta_{\ast})$, provided that $c_{10}$ is chosen small enough. Weiss claims that the same error estimate also holds when $\rho^{\ast}$ does not exist, but the reason given does not appear adequate. (A factor of $\log(d(H) T^n)$ appears to have been overlooked.) However, we can prove a negligibly weaker estimate as follows:
\begin{align*} \log(d(H) T^n) \le \log(Q T^n) &= \log{Q} + 2n \log{A} \\ &\ll \log{Q} + 2n\left(\log(2n) + \log \frac{1}{\delta}\right).  \end{align*}
From the argument at the top of p.\ 90 already alluded to,
\[ (2n)^{2n} Q \cdot h_H A^k T^{1-k} \le \delta. \]
Now $\log(Q) + 2n\log{(2n)} = \log((2n)^{2n} Q) < (2n)^{2n} Q$, and $2n \log\frac{1}{\delta} < (2n)^{2n} Q \log\frac{1}{\delta}$. Hence, $h_H A^k T^{1-k} \cdot \log(d(H) T^n) \ll \delta \log\frac{1}{\delta}$,
so that the second error term above is
\[ O(\delta \log\frac{1}{\delta} \cdot \Delta_{\ast}); \]
we use here that $\Delta_{\ast}$ is a positive constant when $\rho_{\ast}$ does not exist (see the definition of $\Delta_{\ast}$ at the bottom of p.\ 88). Hence, whether or not there is an exceptional zero,
\[ 10n \cdot \frac{h_{H} \log{x}}{\delta x} \cdot \pi_{\Cc}(x,\delta) \ge 1 - \sum_{\chi(H)=1} \sum_{\rho_{\chi}}|h_k(\rho_\chi-1) y^{\rho_{\chi}-1}| - O(\delta \log\frac{1}{\delta} \cdot \Delta_{\ast}). \]

We are assuming that either there is no exceptional zero or that the exceptional character $\chi$ does not satisfy $\chi(H)=1$.
The second paragraph on p.\ 90 shows that under this assumption, the double sum on $\chi$ and $\rho_{\chi}$ is $O(\Delta_{\ast} \exp(-c_1 \Ll^{-1}  \log{y}))$.  Moreover, earlier in the proof (see the very last statement of p.\ 88), it is pointed out that $y \ge \exp(\frac{1}{2} c_{11} \Ll)$. Thus,
$\exp(-c_1 \Ll^{-1} \log{y}) \le \exp(-\frac{1}{2} c_1 c_{11})$.
Inserting this above gives
\[ 10n \cdot \frac{h_H \log{x}}{\delta x} \cdot \pi_{\Cc}(x,\delta) \ge 1 - O(\Delta_{\ast} \exp(-\frac{1}{2} c_{1} c_{11}))- O(\delta \log\frac{1}{\delta}\cdot \Delta_{\ast}). \]
Now $\Delta_{\ast} \ll 1$. If we choose $c_{11}$ sufficiently large, then the first $O$-term will be smaller than $\frac13$ (say). If $c_{10}$ is chosen sufficiently small, then \eqref{eq:deltaineq} forces $\delta$ to be small, and so the second $O$-term will also be smaller than $\frac13$. Hence, $10n \cdot \frac{h_H \log{x}}{\delta x} > \frac{1}{3}$, yielding the theorem.\qedhere\end{owndesc}
\end{proof}

\subsection{A workhorse result} To proceed, we need to modify Theorem \ref{thm:weiss} ever so slightly. Let $\Ss$ be a finite set of nonzero ideals of $\Oo_K$. We can choose a small positive constant $c$ so that none of the finitely many $L$-functions $L(s,\chi)$, corresponding to characters $\chi$ mod $\mm$ with $\mm\in \Ss$, have a real zero $> 1-c$. If we replace $c_1$ with $c_1':=\min\{c,c_1\}$ in Proposition \ref{prop:except0}, it follows automatically that these $L(s,\chi)$ have no exceptional zeros (regardless of the choices of $Q$ and $T$). We call remaining exceptional zeros \emph{exceptional with respect to $Q, T$, and $\Ss$}.

The proof of Theorem \ref{thm:weiss} can now be run as before, but with ``exceptional zero corresponding to $Q$ and $T$'' replaced by ``exceptional zero with respect to $Q, T$, and $\Ss$''. This immediately gives an analogue of Theorem \ref{thm:weiss} that we will call Theorem \ref{thm:weiss}$'$. Note that changing $c_1$ to $c_1'$ has a trickle-down effect, so that in the statement of Theorem 3.2$'$ the  constants $c_{10}$ and $c_{11}$ are replaced by suitable constants $c_{10}'$ and $c_{11}'$ depending on $\Ss$.

We now formulate an important consequence of Theorem \ref{thm:weiss}$'$. For each $\aa \in I(\mm)$, let
\[ \pi(x; \mm, \aa) = \sum_{\substack{\Nm(\pp) \le x \\ \deg(\pp)=1 \\ \pp \sim \aa \pmod{\mm}}} 1. \]
In what follows, we write $h(\mm)$ for $\#I(\mm)/P_{\mm}$. This replaces our previous, more cumbersome notation $h_{P_\mm}$ for the same quantity.

\begin{thm}\label{thm:keythm} Let $K$ be a number field, and let $\Ss$ be a finite set of nonzero ideals of $\Oo_K$. Let $X \ge y^{C_1}$, where $y\ge 2$. Suppose $\Nm(\mm) \le y$ and that $\mm$ is not divisible by the exceptional modulus $\ff_{\chi}$ with respect to to $\Ss$, $Q:= d_K y$, and $T:= C_2 y^{1/n}$ (if it exists). Then
\[ \pi(X;\mm,\aa) \gg \frac{X}{h(\mm) \log{X}}. \]
Here the $C_i$ are positive constants depending on $K$ and $\Ss$, and the final implied constant can also depend on $K$ and $\Ss$.
\end{thm}
\begin{proof} We apply Theorem \ref{thm:weiss}$'$ with $H = P_\mm$, with $\Cc$ the coset of $\aa$ modulo $P_\mm$, with $Q = d_K y$, and with $\delta = C_4 y^{-\frac{1}{2n}}$, for $C_4$ suitably small (to be specified momentarily). We will choose $C_2 = 16(2n+3)^2 C_4^{-2}$; then the exceptional zero hypothesis made in Theorem \ref{thm:keythm} corresponds exactly to that in Theorem \ref{thm:weiss}$'$, since $(4(2n+3) \delta^{-1})^2 =C_2 y^{1/n}$.

Let us check that hypotheses \eqref{eq:deltaineq} and \eqref{eq:xineq} of Theorem \ref{thm:keythm} are satisfied. It is classical (see, e.g., \cite[Proposition 2.1, p.\ 50]{childress09}) that
 \[ h(\mm) = \frac{h \cdot 2^{r_1} \cdot \Phi(\mm)}{[U: U_{\mm}^{+}]}. \]
Here $h$ is the class number of $K$, the group $U$ is the collection of units of $\Oo_K$, and $U_{\mm}^{+}$ is the subgroup of totally positive units congruent to $1$ modulo $\mm$. Thus,
\[ h(\mm) \le h \cdot 2^{r_1}  \Phi(\mm) \le h \cdot 2^{r_1} y.\]
(Recall our assumption that $\Nm(\mm) \le y$.) Also,
\[ d(H) \le d_K \cdot \Nm(\mm) \le d_K y. \]
The quantities $n$, $h$, $r_1$, and $d_K$ are determined by $K$.
So if $C_4$ is chosen suitably small, depending on the field $K$ and the value of $c_{10}'$, then
\[ C_4 y^{-\frac{1}{2n}} \le c_{10}' h(\mm)^{-\frac{1}{2n}} Q^{-\frac{1}{2n}}. \]
Thus, $\delta$ is in the desired range \eqref{eq:deltaineq}.
Turning to \eqref{eq:xineq}, notice that if $C_5$ is chosen sufficiently large in terms of $C_4$, $K$, and $c_{11}'$, then
\[ (30 n Q^{\frac{1}{2n}} \delta^{-1})^{c_{11}' n} \le C_5 y^{c_{11}'}. \]
If $C_1$ is chosen sufficiently large in terms of $C_5$ and $c_{11}'$, then
\[ C_5 y^{c_{11}'} \le \frac{1}{2}y^{C_1}. \]
We can assume that $C_1 \ge 2$, so that
\[ (\log \Nm(\nn))^{2} \le (\log \Nm(\mm))^2 \le (\log{y})^2 \le \frac{1}{2}y^{C_1}. \]
It follows that the hypothesis \eqref{eq:xineq} holds for any $x \ge \frac{1}{2}y^{C_1}$. So by Theorem \ref{thm:weiss}$'$,
\[ \pi_\Cc(x,\delta) \gg \frac{\delta x}{h(\mm) \log{x}}.\]
We have absorbed the factor of $n^{-1}$ into the implied constant, which we remind the reader is now allowed to depend on $K$.

We seek a lower bound on $\pi(X;\mm,\aa)$ rather than a lower bound on primes in short intervals. Thus, we add up the lower bounds on $\pi_{\Cc}(x,\delta)$ over an appropriate set of values of $x$. Let $x_0 = \frac{1}{2} y^{C_1}$, and let $x_j = (1-\delta)^{-j} x_0$. Choose $J$ as large as possible with $x_J \le X$. Then
\begin{multline*} \pi(x;\mm,\aa) \ge \sum_{j=0}^{J}\pi_{\Cc}(x_j, \delta) \gg \frac{\delta x_0}{h(\mm) \log{X}} \sum_{j=0}^{J} (1-\delta)^{-j} \\
 =\frac{\delta x_0}{h(\mm) \log{X}} \cdot \frac{(1-\delta)^{-(J+1)}-1}{(1-\delta)^{-1}-1} \gg \frac{x_0}{h(\mm) \log{x}} ((1-\delta)^{-(J+1)}-1).\end{multline*}
By the choice of $J$, we have
\[ x_0 ((1-\delta)^{-(J+1)}-1) = x_{J+1} - x_0 \ge X-x_0 \ge \frac{1}{2} X, \]
using our assumption that $X \ge y^{C_1}$. Thus, $\pi(X;\mm,\aa) \gg \frac{X}{h(\mm) \log{X}}$.
\end{proof}

\section{The lower bound in Theorem \ref{thm:main}}\label{sec:lower2} We let $E/\Q$ denote a fixed elliptic curve with complex multiplication. We will write $\sum'$ for a sum restricted to primes $p$  of good reduction. By an argument seen earlier,
\begin{equation}\label{eq:sumlower} \sum_{p \le x} d_p = \sum_{d \le 2\sqrt{x}}\phi(d) \sumprime_{\substack{p \le x \\ d \mid d_p}} 1. \end{equation}
Our strategy is to obtain a lower bound for the double sum by carefully estimating the inner sum from below for a sufficiently dense set of values of $d$.

To avoid technical complications, we only consider integers $d > 2$. The primes $p$ of good reduction for which $d$ divides $d_p$ are exactly those that split completely in $\Q(E[d])$ (see \cite[Lemma 2.7]{kowalski06}). Since $d > 2$, we know that $K(E[d]) = \Q(E[d])$ \cite[Lemma 6]{murty83}. Thus, $p$ splits completely in $\Q(E[d])$ if and only if $p$ splits completely in $K$ and the primes of $K$ lying above $p$ split completely in $K(E[d])$. We analyze the  $\pp$ that split completely in $K(E[d])$ by means of the following lemma.

\begin{lem}\label{lem:murty} There is an ideal $\mm$ of $\Oo_K$, depending only on $E$, with the following property: For each positive integer $d$, a prime $\pp$ not dividing $d\mm$ splits completely in $K(E[d])$ if and only if $\pp$ lies in one of $t(d)$ cosets modulo $P_{d\mm}$, where
\[ t(d)  = h(d\mm) \cdot [K(E[d]): K]^{-1}. \]
\end{lem}

\begin{proof} Except for the formula for $t(d)$, this follows from \cite[Lemma 4]{murty83}. From the asymptotic equidistribution of prime ideals mod $P_{\mm}$ (see \cite[Corollary 4, p.\ 349]{narkiewicz04}), the density of $\pp$ splitting completely in $K(E[d])$ is $t(d)/h(d\mm)$. On the other hand, the Chebotarev density theorem implies that this density is also $[K(E[d]):K]^{-1}$. Comparing these two statements gives the stated formula.
\end{proof}

In the following arguments, implied constants may depend on $E$ unless otherwise stated.

Given $d$, we let $\aa_1, \dots, \aa_{t(d)}$ be elements of $I(\mm)$ representing the cosets modulo $P_{d\mm}$ appearing in Lemma \ref{lem:murty}. Piecing the above facts together, we deduce that when $d > 2$,
\[ \sumprime_{\substack{p \le x \\ d\mid d_p \\ p \nmid d\cdot \Nm(\mm)}} 1 = \frac{1}{2} \sumprime_{p \le x}\sum_{\substack{\pp \mid p \\ e(\pp/p) = f(\pp/p) = 1\\ \pp \sim \aa_i\pmod{d\mm}\text{ for some $i$}}}1 = \frac{1}{2} \sum_{i=1}^{t(d)} \pi(x;d\mm,\aa_i) + O(1).\]
Since only $O(\log(2d))$ primes divide $d\cdot \Nm(\mm)$, we conclude that
\begin{equation}\label{eq:summaryconclusion} \sumprime_{\substack{p \le x \\ d\mid d_p}} 1 = \frac{1}{2} \sum_{i=1}^{t(d)} \pi(x;d\mm, \aa_i) + O(\log(2d)). \end{equation}

We apply Theorem \ref{thm:keythm} with $K$ the CM field, $\Ss$ consisting solely of the ideal $\mm$ from Lemma \ref{lem:murty}, $X = x$, and $y = x^{1/C_1}$. If the exceptional modulus $\ff_{\chi}$ exists, then $\ff_{\chi} \nmid \mm$. Hence, there is a prime $\qq$ dividing $\ff_{\chi}$ to a higher power than to which it divides $\mm$. Let $q$ be the rational prime below $\qq$. We obtain a lower bound on $\sum_{p \le x} d_p$ by restricting the final sum on $d$ in \eqref{eq:sumlower} to values
\[ 2 < d \le x^{\frac{1}{2C_1}} \Nm(\mm)^{-1/2} =:Z, \quad\text{with}\quad \text{$d$ coprime to $q$}. \]
From \eqref{eq:summaryconclusion},
\[ \sum_{\substack{ 2 < d \le Z \\ \gcd(d,q)=1}} \phi(d) \sumprime_{\substack{p \le x \\ d\mid d_p}} 1 \\ = \frac{1}{2} \sum_{\substack{2 < d \le Z \\ \gcd(d,q)=1}} \phi(d) \sum_{i=1}^{t(d)} \pi(x;d\mm, \aa_i) + O(x^{1/C_1} \log{x}). \]
Now $C_1$ is a large constant. Hence, the error term is $o(x)$, and so to complete the proof of Theorem \ref{thm:main} it remains only to show that the main term is $\gg x$. For $d$ as above, the modulus $d\mm$ is not divisible by $\ff_{\chi}$, and $\Nm(d\mm) \le y$. By Theorem \ref{thm:keythm},
\[ \sum_{i=1}^{t(d)} \pi(x;d\mm, \aa_i) \gg \frac{t(d)}{h(d\mm)}  \frac{x}{\log{x}} = \frac{1}{[K(E[d]):K]} \frac{x}{\log{x}}. \]
Since $[K(E[d]):K]\ll d^2$, we conclude that
\begin{equation}\label{eq:almostdone2} \sum_{\substack{2 < d \le Z \\ \gcd(d,q)=1}} \phi(d)  \sum_{i=1}^{t(d)} \pi(x;d\mm, \aa_i) \gg \frac{x}{\log{x}} \sum_{\substack{2 < d \le Z \\ \gcd(d,q)=1}} \frac{\phi(d)}{d^2}.  \end{equation}

To show that the final sum on $d$ is $\gg \log{x}$ (for large $x$), we use the following simple observation.

\begin{lem}\label{lem:trivlem} Let $g$ be a multiplicative function taking only nonnegative values. For any positive integer $k$, and any real $t > 0$,
\[ \sum_{\substack{n \le t \\ \gcd(n,t)=1}} \mu^2(n) g(n) \ge \bigg(\prod_{p \mid k}(1+g(p))^{-1}\bigg) \bigg(\sum_{n \le t} \mu^2(n) g(n)\bigg). \]
\end{lem}
\begin{proof} We can factor each squarefree $n\le t$ in the form $n=n_1 n_2$, where $n_1 \mid k$ and $n_2$ is coprime to $k$. Then
\begin{align*} \sum_{n \le t} \mu^2(n) g(n) &\le   \bigg(\sum_{n_1 \mid k} \mu^2(n_1) g(n_1)\bigg) \bigg(\sum_{\substack{n_2 \le t \\ \gcd(n_2,k)=1}} \mu^2(n_2) g(n_2)\bigg) \\&= \bigg(\prod_{p \mid k}(1+g(p))\bigg)\bigg(\sum_{\substack{n_2 \le t \\ \gcd(n_2,k)=1}} \mu^2(n_2) g(n_2)\bigg). \end{align*}
Rearranging yields the result.
\end{proof}

Applying Lemma \ref{lem:trivlem} with $g(n) = \phi(n)/n^2$ and $k=q$,
\begin{align*} \sum_{\substack{2 < d \le Z \\ \gcd(d,q)=1}} \frac{\phi(d)}{d^2} &\ge \sum_{\substack{2 < d \le Z \\ \gcd(d,q)=1}} \mu^2(d) \frac{\phi(d)}{d^2}\\ &\ge \frac{1}{2} \sum_{2 < d \le Z}  \mu^2(d) \frac{\phi(d)}{d^2}. \end{align*}
The multiplicative function $d\mapsto \mu^2(d)\frac{\phi(d)}{d}$ has a well-defined nonzero mean value (for instance, by an elementary theorem of Wintner \cite[Corollary 2.3, p.\ 51]{SS94}). By partial summation, the final displayed sum on $d$ is $\gg \log(Z) \gg \log{x}$, as desired. Inserting this back into \eqref{eq:almostdone2} completes the proof.

\begin{remark} There is no essential difficulty in extending the upper bound half of Theorem \ref{thm:main} to elliptic curves defined over an arbitrary number field $L$. In that case, the sum on $p \le x$ should be replaced with a sum over prime ideals of norm bounded by $x$, and the implied constant may now depend on $L$. We have not yet obtained a corresponding generalization of the lower bound; the obstruction is that we do not know an appropriate analogue of Lemma \ref{lem:murty}.
\end{remark}

\section*{Acknowledgements} The second author would like to express his continuing gratitude to Pete L. Clark for helpful conversations on the theory of elliptic curves. He is supported by NSF award DMS-1402268.

\bibliographystyle{amsalpha}

\providecommand{\bysame}{\leavevmode\hbox to3em{\hrulefill}\thinspace}
\providecommand{\MR}{\relax\ifhmode\unskip\space\fi MR }
\providecommand{\MRhref}[2]{%
  \href{http://www.ams.org/mathscinet-getitem?mr=#1}{#2}
}
\providecommand{\href}[2]{#2}

\end{document}